\documentclass[12pt]{article}
\usepackage{booktabs}
\usepackage{caption}
\usepackage{mathrsfs}
\usepackage{amsmath}
\usepackage{amsfonts,amsthm,amssymb,mathrsfs,bbding}
\usepackage{graphics,multicol}
\usepackage{graphicx}
\usepackage{color}
\usepackage{enumerate}
\usepackage{caption}
\usepackage{rotating}
\usepackage{lscape}
\usepackage{longtable}
\allowdisplaybreaks[4]
\usepackage{tabularx}
\usepackage[colorlinks=true,anchorcolor=blue,filecolor=blue,linkcolor=blue,urlcolor=blue,citecolor=blue]{hyperref}
\usepackage{extarrows}
\usepackage{cite}
\usepackage{latexsym,bm}
\usepackage{mathtools}
\evensidemargin=\oddsidemargin\topmargin=-1.5cm

\newtheorem{thm}{Theorem}
\newtheorem{prob}{Problem}
\newtheorem{lem}[thm]{Lemma}

\newtheorem{fact}{Fact}

\newtheorem{claim}{Claim}
\theoremstyle{definition}
\renewcommand\proofname{\bf Proof}
\addtocounter{section}{0}
\begin{document}

\title{\bf The maximum spectral radius of wheel-free graphs}
\author{{Yanhua Zhao, Xueyi Huang\footnote{Corresponding author.}\setcounter{footnote}{-1}\footnote{\emph{Email address:} yhua030@163.com (Y. Zhao), huangxymath@163.com (X. Huang), huiqiulin@126.com (H. Lin).} \ and Huiqiu Lin}\\[2mm]
\small Department of Mathematics, East China University of Science and Technology, \\
\small  Shanghai 200237, P.R. China}

\date{}
\maketitle
{\flushleft\large\bf Abstract}  A wheel graph is a graph formed by connecting a single vertex to all vertices of a cycle. A graph is called wheel-free if it does not contain any wheel graph as a subgraph. In 2010, Nikiforov  proposed a Brualdi-Solheid-Tur\'{a}n type problem: what is the maximum spectral radius of a graph of order $n$ that does not contain  subgraphs of particular kind.
In this paper, we study the Brualdi-Solheid-Tur\'{a}n type problem for wheel-free graphs, and we determine the maximum  (signless Laplacian) spectral radius of a wheel-free graph of order $n$. Furthermore, we  characterize the extremal graphs.
\begin{flushleft}
\end{flushleft}

\section{Introduction}
Let $G$ be an undirected simple graph with vertex set $V(G)$ and edge set $E(G)$ (denote by $e(G)=|E(G)|$). For any $v\in V(G)$, let $N_k(v)$ denote the set of vertices at distance $k$ from $v$ in $G$. In particular, the vertex subset $N(v)=N_1(v)$ is called the \textit{neighborhood} of $v$, and  $d_v=|N(v)|$ is called the \textit{degree} of $v$. The \textit{adjacency matrix} of $G$ is defined as $A(G)=(a_{u,v})_{u,v\in V(G)}$, where $a_{u,v}=1$ if $uv\in E(G)$, and $a_{u,v}=0$ otherwise. Let $D(G)=\mathrm{diag}(d_v:v\in V(G))$ denote the diagonal matrix of vertex degrees of $G$. Then $Q(G)=D(G)+A(G)$ is called the \textit{signless Laplacian matrix} of $G$. The  (adjacency) \textit{spectral radius} $\rho_A(G)$ and the \textit{signless Laplacian spectral radius} $\rho_Q(G)$ of $G$ are  the largest eigenvalues of $A(G)$ and $Q(G)$, respectively. In addition, for any $n\times n$ matrix $M$ with only real eigenvalues, we always arrange its eigenvalues in a non-increasing order: $\lambda_1(M)\geq \lambda_2(M)\geq \cdots \geq \lambda_n(M)$.

For any $S,T\subseteq V(G)$ with $S\cap T=\emptyset$, let $E(S,T)$ denote the set of edges between $S$ and $T$ in $G$ (denote by $e(S,T)=|E(S,T)|$), and let  $G[S]$ denote the subgraph of $G$ induced by $S$. For any $e\in E(G)$, let $G-e$ denote the graph obtained by deleting $e$ from $G$. Given two graphs $G$ and $H$,  let $G\nabla H$ denote the graph obtained from the disjoint union $G\cup H$ by adding all edges between $G$ and $H$. For any nonnegative integer $k$, let  $kG$ denote the disjoint union of $k$ copies of $G$. As usual, we denote by $K_n$, $P_n$, $C_n$ and $W_n=K_1\nabla C_{n-1}$  the complete graph, the path, the cycle and the wheel graph on $n$ vertices, respectively. Also, we denote by $B_k$ the book graph with $k$-pages, $F_k$ the graph on $2k+1$ vertices  consisting of $k$ triangles which intersect in exactly one common vertex, and  $K_{s,t}$  the complete bipartite graph with  two parts of size $s$ and $t$.

Let $\mathcal{H}$ be a family of graphs. A graph $G$ is called $\mathcal{H}$-\textit{free} if it does not contain any graph of $\mathcal{H}$ as a subgraph.   The \textit{Tur\'{a}n number} of $\mathcal{H}$, denoted by $ex(n,\mathcal{H})$, is the maximum number of edges in an $\mathcal{H}$-free graph of order $n$.  Let $Ex(n,\mathcal{H})$ denote the set of $\mathcal{H}$-free graphs of order $n$ with $ex(n,\mathcal{H})$ edges. To determine $ex(n,\mathcal{H})$ and characterize the graphs in $Ex(n,\mathcal{H})$ for various kinds of $\mathcal{H}$  is a basic problem in extremal graph theory (see \cite{CF,NI5,SI} for surveys). In particular, for $\mathcal{H}=\{W_{2k}\}$, the Simonovits's theorem  (see \cite[Theorem 1, p. 285]{SIM}) implies that $ex(n,\{W_{2k}\})=ex(n,\{K_4\})=\lfloor\frac{n^2}{3}\rfloor$ and $Ex(n,\{W_{2k}\})=\{T_{n,3}\}$ for sufficiently large $n$, where $T_{n,3}$ is the complete $3$-partite graph of order $n$ with part sizes as equal as possible.  In 2013, Dzido \cite{DZ} improved this result to  $n\geq 6k-10$ for $k\geq 3$. For $\mathcal{H}=\{W_{2k+1}\}$, Dzido and Jastrz\c{e}bski \cite{DJ} proved that $ex(n,\{W_5\})=\lfloor\frac{n^2}{4}+\frac{n}{2}\rfloor$, $ex(n,\{W_7\})=\lfloor\frac{n^2}{4}+\frac{n}{2}+1\rfloor$, and  $ex(n,\{W_{2k+1}\})\geq \lfloor\frac{n^2}{4}+\frac{n}{2}\rfloor$ for all values of $n$ and $k$. Very recently, Yuan \cite{YU} established that $ex(n,\{W_{2k+1}\})=\max\{n_0n_1+\lfloor\frac{(k-1)n_0}{2}\rfloor􏰉+2:n_0+n_1=n\}$ for $k\geq 3$ and sufficiently large $n$.

In spectral graph theory, the well-known Brualdi-Solheid problem (see \cite{BS}) asks for the maximum spectral radius of a graph belonging to a specified class of graphs and the characterization of the extremal graphs. Up to now, this problem has been studied for many classes of graphs, and readers are referred  to  \cite{STE} for  systematic results. As the blending of  the Brualdi-Solheid problem and the general Tur\'{a}n type problem, Nikiforov \cite{NI3} proposed a Brualdi-Solheid-Tur\'{a}n type problem: 
\begin{prob}\label{prob-2}
What is the maximum spectral radius of an $\mathcal{H}$-free graph of order $n$?
\end{prob}

In the past few decades, much attention has been paid to Problem \ref{prob-2}  for various families of graphs $\mathcal{H}$ such as  $\mathcal{H}=\{K_s\}$ \cite{NI1,WI}, $\{K_{s,t}\}$ \cite{BG,NI1,NI4}, $\{{B_{k+1},K_{2,l+1}}\}$ \cite{SS}, $\{F_k\}$ \cite{CFTZ} $\{P_s\}$ \cite{NI3}, $\{C_{2k+1}\}$ \cite{NI2}, $\{C_3,C_4\}$ \cite{LLT}, $\{C_4\}$ \cite{NI1,ZW}, $\{C_5,C_6\}$ \cite{YWZ}, $\{W_5,C_6\}$ \cite{ZWF}, $\{C_6\}$ \cite{ZL},  $\{\cup_{i=1}^kP_{s_i}\}$ \cite{CLZ}, $\{C_{l}:l\geq 2k+1\}$ and $\{C_{l}:l\geq 2k+2\}$ \cite{GH}. For more  results on  extremal spectral graph theory, we refer the reader to \cite{NI5}.  

Motived by the general Tur\'{a}n type problem, it is natural to consider the Brualdi-Solheid-Tur\'{a}n type problem  for $\{W_\ell\}$-free graphs, where $l$ is any fixed integer. However,  it seems difficult to determine the maximum spectral radius of a $\{W_\ell\}$-free graph of order $n$. In this paper, we consider a closely related problem:
\begin{prob}\label{prob-3}
What is the maximal spectral radius of a wheel-free (i.e., $\{W_\ell:\ell\geq 4\}$-free) graph of order $n$?
\end{prob}

Let $H_n$ be defined as
\begin{equation}\label{eq-1}
H_n=\left\{
\begin{array}{ll}
\frac{n-1}{4}K_2 \nabla  \frac{n+1}{2}K_1 & \mbox{if $n\equiv 1~\mathrm{mod}~4$}, \\
\frac{n+1}{4}K_2 \nabla  \frac{n-1}{2}K_1 & \mbox{if $n\equiv 3~\mathrm{mod}~4$}, \\
\frac{n}{4}K_2 \nabla  \frac{n}{2}K_1 & \mbox{if $n\equiv 0~\mathrm{mod}~4$}, \\
(\frac{n-2}{4}K_2\cup K_1) \nabla  \frac{n}{2}K_1 & \mbox{if $n\equiv 2~\mathrm{mod}~4$},
\end{array}
\right.
\end{equation}
and let $F$ be the complement of $C_7$ shown in Figure \ref{fig-1}. Notice that both $H_n$ and $F$ are wheel-free. As an answer to Problem \ref{prob-3}, we prove that
\begin{thm}\label{thm-1}
Let $G$ be a wheel-free graph of order $n\geq 4$. Then 
$$\rho_A(G) \leq \rho_A(H_n),$$ with equality holding if and only if $G=H_n$ for $n\neq7$ and $G=H_7$ or $F$ for $n=7$.
\end{thm}
\begin{figure}[t] 
\centering 
\includegraphics[width=4cm]{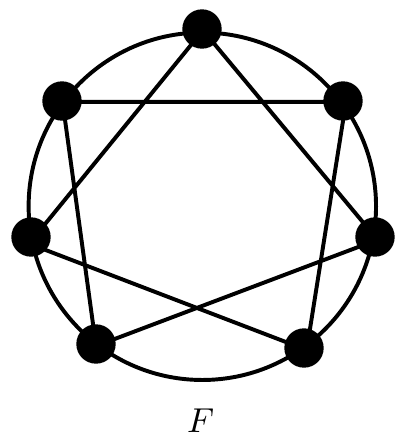} 
\caption{The graph $F$.} 
\label{fig-1}
\end{figure}

Furthermore, we  consider the same problem for the signless Laplacian spectral radius of wheel-free graphs.  Surprisingly,  the following result shows that the extremal graphs are not the same as that of Theorem \ref{thm-1}.

\begin{thm}\label{thm-2}
Let $G$ be a wheel-free graph of order $n \geq 4$. Then 
$$\rho_Q(G) \leq \rho_Q(K_2\nabla (n-2)K_1)=\frac{n+2+\sqrt{(n+2)^2-16}}{2},$$ 
with equality holding if and only if $G=K_2\nabla (n-2)K_1$.
\end{thm}

\section{Some lemmas}
Let $M$ be a real symmetric matrix of order $n$, and let $[n]=\{1,2,\ldots,n\}$.  Given a partition $\Pi:[n]=X_1\cup X_2\cup \cdots \cup X_k$, the matrix $M$ can be written as 
$$
M=\begin{bmatrix}
M_{1,1}&M_{1,2}&\cdots &M_{1,k}\\
M_{2,1}&M_{2,2}&\cdots &M_{2,k}\\
\vdots& \vdots& \ddots& \vdots\\
M_{k,1}&M_{k,2}&\cdots &M_{k,k}\\
\end{bmatrix}.
$$
If  $M_{i,j}$  has constant row sum $b_{i,j}$ for all $i,j\in\{1,2,\ldots,k\}$, then $\Pi$ is called an \textit{equitable partition} of $M$, and the matrix $B_\Pi=(b_{i,j})_{i,j=1}^k$ is called an \textit{equitable quotient matrix}  of $M$.

\begin{lem}(Brouwer and Haemers \cite[p. 30]{BH}; Godsil and Royle \cite[pp. 196--198]{GR}.)\label{quotient}
Let $M$ be a real symmetric matrix, and let $B_\Pi$ be an equitable quotient matrix of $M$. Then the eigenvalues of  $B_\Pi$ are also eigenvalues of $M$. Furthermore, if $M$ is nonnegative and irreducible, then 
$$\lambda_1(M) = \lambda_1(B_\Pi).$$
\end{lem}

The following result is straightforward, and one can find a short proof in \cite{EZ}. 
\begin{lem}\label{row_sum}
Let $M$ be a real symmetric matrix with row sums $R_1,R_2,\ldots,R_n$.  Let $\lambda(M)$
be an eigenvalue of $M$ with an eigenvector $x$ all of whose entries are nonnegative. Then
$$\min_{1\leq i \leq n} R_i \leq \lambda(M) \leq \max_{1\leq i\leq n} R_i.$$
Moreover, if all entries of $x$ are positive, then either of the equalities holds if and only if the row sums of
$M$ are all equal.
\end{lem}

\begin{lem}\label{extremal-graph}
Let  $H_n$ be the graph defined in (\ref{eq-1}) with $n\geq 4$. Let $\lambda_1$ denote the largest root of $x^3 - x^2 - \frac{n^2}{4}x + \frac{n}{2}=0$.  Then 
\begin{equation*}
\rho_A(H_n)=\left\{
\begin{array}{ll}
\frac{n+1}{2} & \mbox{if $n\equiv 1~\mathrm{mod}~4$}, \\
\frac{n+1}{2}& \mbox{if $n\equiv 3~\mathrm{mod}~4$}, \\
\frac{\sqrt{n^2+1}+1}{2} & \mbox{if $n\equiv 0~\mathrm{mod}~4$}, \\
\lambda_1>\frac{\sqrt{n^2 - 3}+1}{2} & \mbox{if $n\equiv 2~\mathrm{mod}~4$}.
\end{array}
\right.
\end{equation*}
\end{lem}
\begin{proof}
If $n\equiv 1~\mathrm{mod}~4$, we see that $A(H_n)$ has the equitable quotient matrix 
$$
B_\Pi=\begin{bmatrix}
1& \frac{n+1}{2}\\
\frac{n-1}{2}& 0
\end{bmatrix}.
$$
Then, by  Lemma \ref{quotient}, we  have $\rho_A(H_n)=\lambda_1(B_\Pi)= \frac{n+1}{2}$.  Similarly, for $n\equiv 3~\mathrm{mod}~4$ and $n\equiv 0~\mathrm{mod}~4$, we have $\rho_A(H_n)= \frac{n+1}{2}$ and $\rho_A(H_n)= \frac{\sqrt{n^2+1}+1}{2}$, respectively. For $n\equiv 2~\mathrm{mod}~4$, observe that $A(H_n)$ has the equitable quotient matrix 
$$
B_\Pi=\begin{bmatrix}
1 & 0 & \frac{n}{2}\\
0 & 0 & \frac{n}{2}\\
\frac{n}{2}-1& 1 & 0\\
\end{bmatrix}.
$$
By a simple calculation, the characteristic polynomial of $B_\Pi$ is equal to 
$$\varphi(B_\Pi,x)=x^3 - x^2 - \frac{n^2}{4}x + \frac{n}{2}.$$
Since $\varphi(B_\Pi, 
\frac{\sqrt{n^2 - 3}+1}{2})=\frac{n-\sqrt{n^2 - 3}-1}{2}<0$ due to $n\geq 4$, we have $\lambda_1(B_\Pi)>\frac{\sqrt{n^2 - 3}+1}{2}$, which gives that  $\rho_A(H_n)=\lambda_1(B_\Pi)>\frac{\sqrt{n^2 - 3}+1}{2}$ 
again by Lemma \ref{quotient}.
\end{proof}

\begin{lem}\label{Q-extremal-graph}
For $n\geq 3$, we have
$$\rho_Q(K_2\nabla (n-2)K_1)=\frac{n+2+\sqrt{(n+2)^2-16}}{2}.$$
\end{lem}
\begin{proof}
Notice that $Q(K_2\nabla (n-2)K_1)$ has the equitable quotient matrix
$$
B_\Pi=\left(\begin{array}{ccccccc}
n & n-2\\
2& 2\\
\end{array}\right).
$$
Then, by Lemma \ref{quotient}, we have 
$$\rho_Q(K_2\nabla (n-2)K_1)=\lambda_1(B_\Pi)=\frac{n+2+\sqrt{(n+2)^2-16}}{2},$$
as required.
\end{proof}

\section{Proofs of Theorems \ref{thm-1} and \ref{thm-2}}

For a wheel-free graph $G$, the following two facts are obvious.
\begin{fact}\label{fact-1}
For any $v \in V(G)$, $G[N(v)]$ is a forest.
\end{fact}

\begin{fact}\label{fact-2}
For any two distinct $u, v\in V(G)$,  $G[N(u)\cap N(v)]$ is  $P_3$-free. Furthermore, if  $uv\in E(G)$ then $G[N(u)\cap N(v)]$ is $K_2$-free.
\end{fact}

First we shall give the proof  of Theorem \ref{thm-1}. 

\renewcommand\proofname{\bf Proof of Theorem \ref{thm-1}}
\begin{proof}
Notice that  both  $H_n$ and $F$  are wheel-free. By using Sagemath v9.1\footnote{The code was uploaded to \url{https://github.com/XueyiHuang/Wheel-free-Graph.git}.} \cite{ST}, we find that,  for $n\leq 10$ and $n\neq 7$,  $H_n$  is the unique graph attaining the maximum spectral radius among all wheel-free graphs of order $n$, and for $n=7$, there is another extremal graph $F$ (see Figure \ref{fig-1}), which satisfies $\rho_A(F)=\rho_A(H_7)=4$. From now on, we  take $n\geq 11$ and  assume that $G$ is a wheel-free graph of order $n$ with maximum spectral radius.  We assert that $G$ is connected. If not, suppose that $G_1,\ldots,G_\omega$ are the components of $G$. Then we can add $\omega-1$ edges to $G$ so that the obtained graph $G^*$ is connected and wheel-free. By using the Rayleigh quotient and the Perron-Frobenius theorem, we can deduce that $\rho_A(G^*)>\rho_A(G)$, contrary to the assumption.

For any $v\in V(G)$, we denote by $\omega_v$ the number of components  in $G[N(v)]$, $\bar{d}_v=|N_2(v)|$ the number of vertices at distance $2$ from $v$, and $R_v$  the row sum of $A(G)^2$ corresponding to $v$.  Notice that $R_v$ is exactly the number of walks of length $2$  originating at  $v$.  Thus 
\begin{equation}\label{eq-2}
R_v=d_v+2e(G[N(v)])+e(N(v),N_2(v))
\end{equation}
for any $v\in V(G)$. Take $u\in V(G)$ such that $R_u=\max_{v\in V(G)}R_v$.
We have the following three claims.
\begin{claim}\label{claim-1}
$R_u\geq \frac{(n+1)^2-1}{4}$.
\end{claim}
\renewcommand\proofname{\bf Proof}
\begin{proof}
By assumption, we find that $\rho_A(G)\geq \rho_A(H_n)$. If $n\not\equiv 2~\mathrm{mod}~4$, from  
the Perron-Frobenius theorem, Lemma \ref{row_sum} and Lemma \ref{extremal-graph} we immediately deduce that 
$$R_u\geq \rho_A^2(G)\geq \rho_A^2(H_n)\geq \frac{(n+1)^2}{4}.$$ Similarly, for $n\equiv 2~\mathrm{mod}~4$, i.e.,  $n=4k+2$ with $k\in \mathbb{Z}$,  we have 
$$
R_u>\frac{(\sqrt{n^2 - 3}+1)^2}{4}>\frac{(n+1)^2}{4}-1=4k^2+6k+\frac{5}{4}.
$$
Since $R_u$ is an integer, we conclude that 
$$
R_u\geq 4k^2+6k+2=\frac{(n+1)^2-1}{4}.
$$
This proves Claim \ref{claim-1}.
\end{proof}

\begin{claim}\label{claim-2}
$\bar{d}_u=n-1-d_u$, or equivalently, $V(G)=\{u\}\cup N(u)\cup N_2(u)$.
\end{claim}
\begin{proof}
Notice that $e(G[N(u)])=d_u-\omega_u$ by Fact \ref{fact-1}. According to (\ref{eq-2}), we have
$$
\begin{aligned}
R_u&\leq d_u+2(d_u-\omega_u)+d_u\bar{d}_u\leq d_u(\bar{d}_u+3)-2\leq \frac{(d_u+\bar{d}_u+3)^2}{4}-2.
\end{aligned}
$$
Combining this with Claim \ref{claim-1}, we get
$$
(n+1)^2<(d_u+\bar{d}_u+3)^2\leq (n+2)^2,
$$
which implies that $d_u+\bar{d}_u=n-1$ because $d_u+\bar{d}_u$ is an integer. 
\end{proof}

\begin{claim}\label{claim-3}
Let $p_u$ be the number of vertex-disjoint copies of $P_3$ in $G[N(u)]$. We have $p_u\leq 1.$
\end{claim}
\begin{proof}
By Fact \ref{fact-2},  each vertex (if any) of $N_2(u)$ is adjacent to at most two vertices of any $P_3$ of $G[N(u)]$. Thus, if $p_u\geq 2$, we have
$$
\begin{aligned}
R_u&\leq d_u +2(d_u-\omega_u)+(d_u-p_u)\bar{d}_u\\
&\leq d_u +2(d_u-1)+(d_u-2)\bar{d}_u\\
&=-d_u^2+(n+4)d_u-2n
\end{aligned}
$$
by (\ref{eq-2}) and Claim \ref{claim-2}. Combining the above inequality with Claim \ref{claim-1} yields that $$d_u^2-(n+4)d_u+2n+\frac{(n+1)^2-1}{4}\leq 0,$$ which is impossible because $\Delta=(n+4)^2-4(2n+\frac{(n+1)^2-1}{4})=16-2n<0$. Thus we must have $p_u\leq 1$.
\end{proof}

By Claim \ref{claim-3}, it suffices to consider the following two cases.

{\flushleft \bf Case 1.} $p_u=1.$

\begin{figure}[t] 
\centering 
\includegraphics[width=6cm]{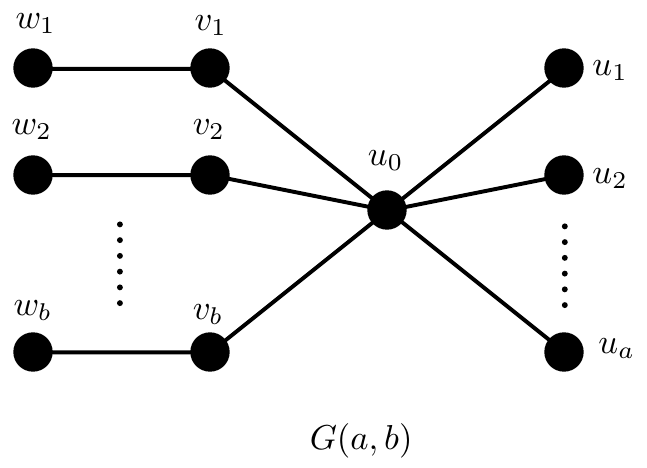} 
\caption{The graph $G(a,b)$.} 
\label{fig-2}
\end{figure}

In this case, we see that  $G$ has the  following six properties:
\begin{enumerate}[(P1)]
\item $d_u\in\{\frac{n+1}{2},\frac{n+3}{2},\frac{n+5}{2}\}$ when $n$ is odd, and  $d_u\in\{\frac{n+2}{2},\frac{n+4}{2}\}$ when $n$ is even;
\item $G[N(u)]=G(a,b)$ (see Figure \ref{fig-2}), where $a,b\geq 0$ and $a+2b+1=d_u\geq 6$;
\item $G[N_2(u)]=cK_2\cup dK_1$, where $c,d\geq 0$ and  $2c+d=\bar{d}_u$;
\item $e(N(u),N_2(u))=d_u\bar{d}_u-\bar{d}_u$ or $d_u\bar{d}_u-\bar{d}_u-1$;
\item for any $v\in N_2(u)$ and $P_3\subseteq G[N(u)]$, we have $|N(v)\cap V(P_3)|=1$ or $2$,  and for any fixed $P_3\subseteq G[N(u)]$, there is at most one  $v\in N_2(u)$ such that $|N(v)\cap V(P_3)|=1$;
\item each vertex of  $N_2(u)$ is not adjacent to  $u_0$,  where $u_0$ is the central vertex of $G[N(u)]=G(a,b)$ shown in Figure \ref{fig-2}.
\end{enumerate}

For (P1), by Fact \ref{fact-1}, Fact \ref{fact-2} and Claim \ref{claim-2},  we get
$$
\begin{aligned}
R_u&=d_u+2e(G[N(u)])+e(N(u),N_2(u))\\
&\leq d_u +2(d_u-1)+d_u\bar{d}_u-\bar{d}_u\\
&=-d_u^2+(n+3)d_u-n-1,
\end{aligned}
$$
Combining this with Claim \ref{claim-1}, we obtain $\frac{n+3-\sqrt{5}}{2}\leq d_u\leq \frac{n+3+\sqrt{5}}{2}$. Considering that $d_u$ is an integer, we may conclude that $d_u\in\{\frac{n+1}{2},\frac{n+3}{2},\frac{n+5}{2}\}$ when $n$ is odd, and $d_u\in\{\frac{n+2}{2},\frac{n+4}{2}\}$ when $n$ is even. For (P2), if $G[N(u)]$ is disconnected, i.e., $\omega_u\geq 2$, we have 
$$
\begin{aligned}
R_u&=d_u+2(d_u-\omega_u)+e(N(u),N_2(u))\\
&\leq d_u +2(d_u-2)+d_u\bar{d}_u-\bar{d}_u\\
&=-d_u^2+(n+3)d_u-n-3.
\end{aligned}
$$
Combining this with Claim \ref{claim-1} yields that
$$
d_u^2-(n+3)d_u+n+3+\frac{(n+1)^2-1}{4}\leq 0,
$$
which is impossible because $\Delta=(n+3)^2-4(n+3+\frac{(n+1)^2-1}{4})=-3<0$. Thus $G[N(u)]$ is a tree by Fact \ref{fact-1}.  Furthermore, since $G[N(u)]$ has exactly one vertex-disjoint copy of $P_3$, we immediately deduce that $G[N(u)]=G(a,b)$ (see Figure \ref{fig-2}), where $a,b$ are  nonnegative integers such that $a+2b+1=d_u$. Notice that $d_u\geq 6$ follows from (P1) and $n\geq 11$. For (P3), it suffices to prove that $G[N_2(u)]$ is $P_3$-free. By contradiction, assume that there exists some copy of $P_3$ (say $v_1v_2v_3$) in $G[N_2(u)]$. As $p_u=1$, we also can take a copy of $P_3$ (say $u_1u_2u_3$) in $G[N(u)]$.  By Fact \ref{fact-2}, each vertex of $N_2(u)$ is adjacent to at most two vertices of $\{u_1,u_2,u_3\}$, and there is at most one  $w\in N(u)\setminus\{u_1,u_2,u_3\}$ such that $\{v_1,v_2,v_3\}\subseteq N(w)$. Thus we have $e(N(u),N_2(u))\leq d_u\bar{d}_u-\bar{d}_u-(d_u-4)$, and 
$$
R_u\leq d_u +2(d_u-1)+d_u\bar{d}_u-\bar{d}_u-(d_u-4)=-d_u^2+(n+2)d_u-n+3.
$$
by Fact  \ref{fact-1} and Claim \ref{claim-2}. 
Combining this with Claim \ref{claim-1} yields that
$$
d_u^2-(n+2)d_u+n-3+\frac{(n+1)^2-1}{4}\leq 0,
$$
which is impossible because $\Delta=(n+2)^2-4(n-3+\frac{(n+1)^2-1}{4})=-2n+16<0$. This proves (P3). For (P4), if $e(N(u),N_2(u))\leq d_u\bar{d}_u-\bar{d}_u-2$, as in (P2), we also can deduce a contradiction. Thus the result follows because we have known  that $e(N(u),N_2(u))\leq d_u\bar{d}_u-\bar{d}_u$ by Fact \ref{fact-2}. For (P5),  it is clear that $|N(v)\cap V(P_3)|\leq 2$ by Fact \ref{fact-2}. Also, if  $|N(v)\cap V(P_3)|=0$ or there are two vertices $v_1,v_2\in V(N_2(u))$ such that $|N(v_1)\cap V(P_3)|=|N(v_2)\cap V(P_3)|=1$, then $e(N(u),N_2(u))\leq d_u\bar{d}_u-\bar{d}_u-2$, which contradicts (P4). For (P6), suppose to the contrary that  there exists some $v\in N_2(u)$ such that $u_0\in N(v)$. If $a=0$, then $b=\frac{d_u-1}{2}$, and we have  $|N(v)\cap N(u)|\leq b+1=\frac{d_u+1}{2}$  by Fact \ref{fact-2}. Hence, 
$$
\begin{aligned}
e(N(u),N_2(u))&\leq d_u\bar{d}_u-(\bar{d}_u-1)-(d_u-|N(v)\cap N(u)|)\\
&\leq d_u\bar{d}_u-\bar{d}_u-\frac{1}{2}(d_u-3)\\
&<d_u\bar{d}_u-\bar{d}_u-1,
\end{aligned}
$$
contrary to (P4). Similarly, if $a\geq 1$,  then $|N(v)\cap N(u)|\leq b+2\leq \frac{d_u}{2}+1$ and 
$$
e(N(u),N_2(u))\leq d_u\bar{d}_u-(\bar{d}_u-1)-(d_u-|N(v)\cap N(u)|)\leq d_u\bar{d}_u-\bar{d}_u-\frac{d_u}{2}+2.
$$
Thus we have 
$$
\begin{aligned}
R_u&\leq d_u +2(d_u-1)+d_u\bar{d}_u-\bar{d}_u-\frac{d_u}{2}+2=-d_u^2+\frac{2n+5}{2}d_u-n+1.
\end{aligned}
$$
Combining this with Claim \ref{claim-1}, we obtain
$$
\begin{aligned}
d_u^2-\frac{2n+5}{2}d_u+n-1+\frac{(n+1)^2-1}{4}\leq 0,
\end{aligned}
$$
which is impossible because $\Delta=(\frac{2n+5}{2})^2-4(n-1+\frac{(n+1)^2-1}{4})=\frac{41}{4}-n<0$ due to  $n\geq 11$. This proves (P6).

According to (P2)--(P6), we see that $G$ must be of the form $G(a,b,c,d)$ or $G(a,b,c,d)-e$,  where $G(a,b,c,d)$ is shown in Figure \ref{fig-3}, and $e$ is some edge between $N(u)$ and $N_2(u)$ in $G(a,b,c,d)$. Notice that $\rho_A(G(a,b,c,d))>\rho_A(G(a,b,c,d)-e)$.  We consider the following three situations.

\begin{figure}[t] 
\centering 
\includegraphics[width=12cm]{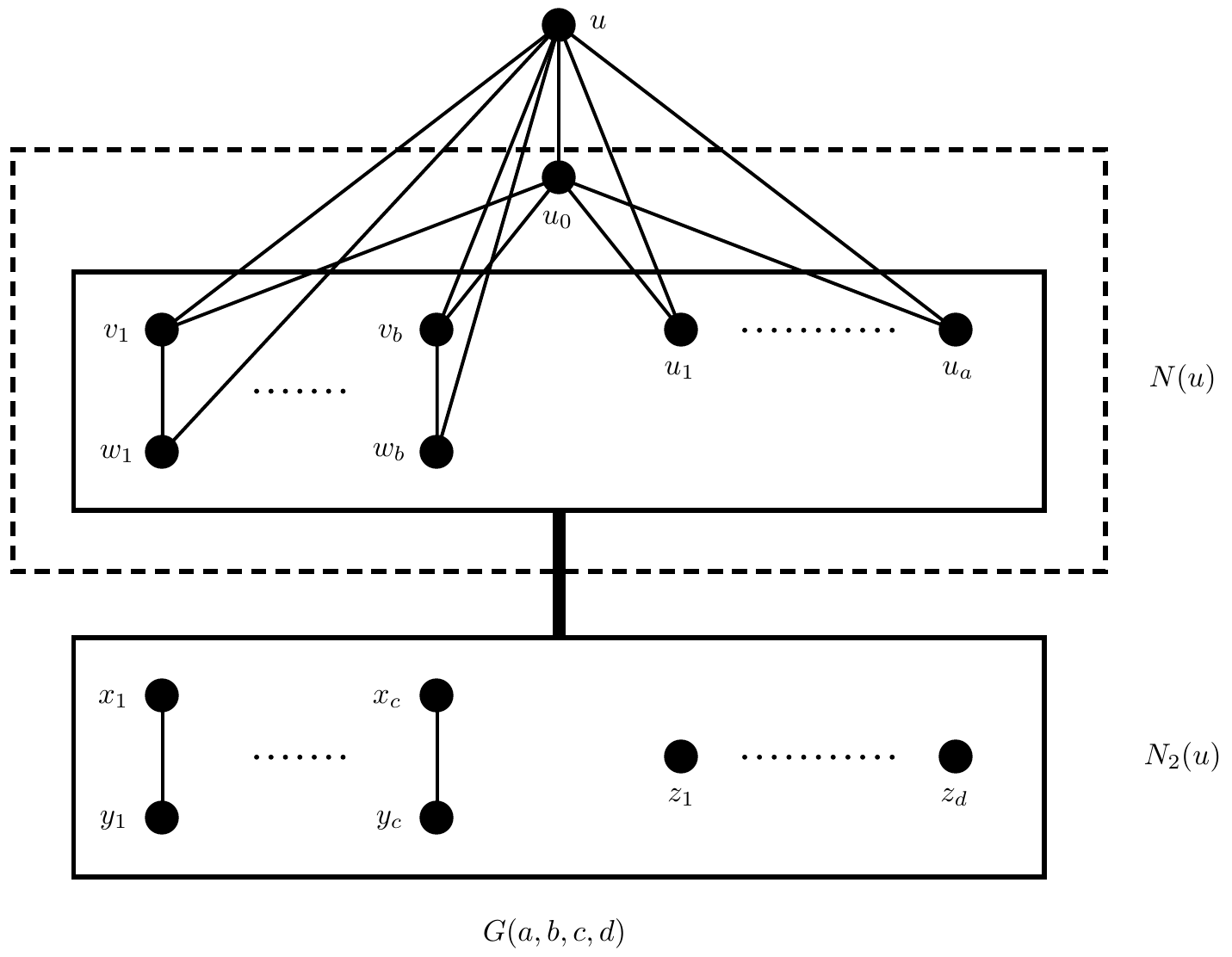} 
\caption{The graph $G(a,b,c,d)$, where the thickest line  represents the connection of  all edges between  $N(u)\setminus\{u_0\}$ and $N_2(u)$.} 
\label{fig-3}
\end{figure}

{\flushleft \bf Subcase 1.1.} $b=0$.

 In this situation,  we have $G=G(d_u-1,0,c,d)$ or $G(d_u-1,0,c,d)-e$, which are both wheel-free. 
Thus we conclude that  $G=G(d_u-1,0,c,d)$ by considering the fact that $G$ has the maximum spectral  radius among all wheel-free graphs of order $n$. For the same reason, we assert that $G=G(d_u-1,0,\lfloor\frac{\bar{d}_u}{2}\rfloor,\bar{d}_u-2\lfloor\frac{\bar{d}_u}{2}\rfloor)=(\lfloor\frac{\bar{d}_u+2}{2}\rfloor K_2\cup (\bar{d}_u-2\lfloor\frac{\bar{d}_u}{2}\rfloor) K_1)\nabla (d_u-1)K_1$.  Recall that $\bar{d}_u=n-1-d_u$. Then, by (P1), we can determine all possible forms of $G$, which are listed in Table \ref{tab-1}.

For $n\equiv 1~\mathrm{mod}~4$, we shall  prove that 
\begin{equation}\label{eq-3}
\left\{
\begin{aligned}
&\rho_A\Big(\Big(\frac{n-1}{4}K_2\cup K_1\Big)\nabla \frac{n-1}{2}K_1\Big)<\rho_A(H_n), \\&\rho_A\Big(\Big(\frac{n-5}{4}K_2\cup K_1\Big)\nabla \frac{n+3}{2}K_1\Big)<\rho_A(H_n).
\end{aligned}
\right.
\end{equation}
Observe that $A((\frac{n-1}{4}K_2\cup K_1)\nabla \frac{n-1}{2}K_1)$ has the  equitable quotient matrix
$$
B_\Pi=\begin{bmatrix}
1 & 0 &\frac{n-1}{2}\\
 0 & 0 & \frac{n-1}{2}\\
 \frac{n-1}{2} &1 &0
\end{bmatrix},
$$
of which the characteristic polynomial is equal to
$$
\varphi(B_\Pi,x)=x^3 - x^2 -\frac{n^2-1}{4}x + \frac{n-1}{2}.
$$
As $\varphi(B_\Pi,\frac{n+1}{2})=\frac{n-1}{2}>0$, we have $\lambda_1(B_\Pi)<\frac{n+1}{2}$ or $\lambda_2(B_\Pi)>\frac{n+1}{2}$. We claim that the later case cannot occur, since otherwise we have $\lambda_3(B_\Pi)<-n$ by considering the trace of $B_\Pi$, which is impossible because $\varphi(B_\Pi,-n)=- \frac{3}{4}n^3 - n^2 + 
\frac{1}{4}n - \frac{1}{2}<0$. It follows that $\rho_A((\frac{n-1}{4}K_2\cup K_1)\nabla \frac{n-1}{2}K_1)=\lambda_1(B_\Pi)<\frac{n+1}{2}=\rho_A(H_n)$ by Lemma \ref{extremal-graph}. Similarly, one can verify that $\rho_A((\frac{n-5}{4}K_2\cup K_1)\nabla \frac{n+3}{2}K_1)<\rho_A(H_n)$.  Thus  (\ref{eq-3}) holds, and $G=H_n$ by the maximality of $\rho_A(G)$.  For  $n\not\equiv 1~\mathrm{mod}~4$, by using a similar method, we find that  $H_n$ always has the maximum spectral radius. Therefore, we conclude that $G=H_n$ in this situation.

\begin{table}[t]
\caption{All possible forms of $G=(\lfloor\frac{\bar{d}_u+2}{2}\rfloor K_2\cup (\bar{d}_u-2\lfloor\frac{\bar{d}_u}{2}\rfloor) K_1)\nabla (d_u-1)K_1$.}
\centering
\setlength{\tabcolsep}{9mm}{
\begin{tabular}{cccc}
\toprule
$n~\mathrm{mod}~4$ & $d_u$ & $\bar{d}_u$ & $G$\\
\midrule
$1$ & $\frac{n+1}{2}$ & $\frac{n-3}{2}$& $(\frac{n-1}{4}K_2\cup K_1)\nabla \frac{n-1}{2}K_1$\\
$1$ & $\frac{n+3}{2}$ & $\frac{n-5}{2}$ &$\frac{n-1}{4}K_2\nabla \frac{n+1}{2}K_1=H_n$ \\
$1$ & $\frac{n+5}{2}$ & $\frac{n-7}{2}$ & $(\frac{n-5}{4}K_2\cup K_1)\nabla \frac{n+3}{2}K_1$\\
$3$ & $\frac{n+1}{2}$ & $\frac{n-3}{2}$& $\frac{n+1}{4}K_2\nabla \frac{n-1}{2}K_1=H_n$\\
$3$ & $\frac{n+3}{2}$ & $\frac{n-5}{2}$ &$(\frac{n-3}{4}K_2\cup K_1)\nabla \frac{n+1}{2}K_1$\\
$3$ & $\frac{n+5}{2}$ & $\frac{n-7}{2}$ & $\frac{n-3}{4}K_2\nabla \frac{n+3}{2}K_1$\\
$0$ & $\frac{n+2}{2}$ & $\frac{n-4}{2}$& $\frac{n}{4}K_2\nabla \frac{n}{2}K_1=H_n$\\
$0$ & $\frac{n+4}{2}$ & $\frac{n-6}{2}$ & $(\frac{n-4}{4}K_2\cup K_1)\nabla \frac{n+2}{2}K_1$\\
$2$ & $\frac{n+2}{2}$ & $\frac{n-4}{2}$ & ($\frac{n-2}{4}K_2\cup K_1)\nabla \frac{n}{2}K_1=H_n$\\
$2$ & $\frac{n+4}{2}$ & $\frac{n-6}{2}$& $\frac{n-2}{4}K_2\nabla \frac{n+2}{2}K_1$\\
\bottomrule
\end{tabular}}
\label{tab-1}
\end{table}

{\flushleft \bf Subcase 1.2.} $b=1$.

In this situation, we have $G=G(d_u-3,1,c,d)$ or  $G(d_u-3,1,c,d)-e$, where $e$ is some edge between $N(u)$ and $N_2(u)$ in $G(d_u-3,1,c,d)$. Because  $G$ is wheel-free, we must have  $c\leq 1$. 

If $c=0$,  since both $G(d_u-3,1,c,d)$ and  $G(d_u-3,1,c,d)-e$ are wheel-free, we conclude that $G=G(d_u-3,1,0,\bar{d}_u)=G(d_u-3,1,0,n-1-d_u)$ by the maximality of $\rho_A(G)$. Observe that $A(G)$ has the equitable quotient matrix 
$$
B_\Pi=\begin{bmatrix}
0 &1& d_u-3 &1 &1& 0\\
 1 &0 &d_u-3 &1 &0 &0\\
  1 &1& 0 &0& 0& n-1-d_u\\
   1 &1 &0 &0 &1 &n-1-d_u\\
    1 &0 &0& 1& 0 &n-1-d_u\\
     0 &0 &d_u-3 &1 &1 &0
\end{bmatrix}
\begin{array}{l}
\{u\}\\
\{u_0\}\\
\{u_1,\ldots,u_{d_u-3}\}\\
\{v_1\}\\
\{w_1\}\\
\{z_1,\ldots,z_{n-1-d_u}\}
\end{array}
$$
By a simple computation, the characteristic polynomial of $B_\Pi$ is equal to 
$$
\begin{aligned}
\varphi(B_\Pi,x,d_u)&=x^6 + (d_u^2 - (n +2)d_u + n)x^4 + (4 - 2n)x^3 \\
&~~~- (3d_u^2 - (3n + 6)d_u + 6n + 3)x^2+ (2n - 8)x\\
&~~~+ d_u^2 - (n + 2)d_u + 3n - 3.\\
\end{aligned}
$$
Notice that 
$$
\varphi\Big(B_\Pi,\frac{2n+1}{4},d_u\Big)=\alpha(n)\cdot d_u^2 + \beta(n)\cdot d_u +\gamma(n), 
$$
where 
$$
\left\{
\begin{aligned}
\alpha(n)&=\frac{1}{16}n^4+ \frac{1}{8}n^3 - \frac{21}{32}n^2 - \frac{21}{32}n + \frac{209}{256},\\
\beta(n)&= - \frac{1}{16}n^5 - \frac{1}{4}n^4 + \frac{13}{32}n^3 + \frac{65}{32}n^2 + \frac{159}{256}n - \frac{209}{128},\\
\gamma(n)&=\frac{1}{64}n^6 + \frac{7}{64}n^5 - \frac{17}{256}n^4 - \frac{159}{128}n^3 - \frac{657}{1024}n^2 - \frac{1305}{1024}n - \frac{20991}{4096}.
\end{aligned}
\right.
$$
Since $\alpha(n)>0$ due to $n\geq 11$, and $-\beta(n)/(2\alpha(n))=\frac{n+2}{2}$, we obtain 
$$
\begin{aligned}
\varphi\Big(B_\Pi,\frac{2n+1}{4},d_u\Big)&\geq \varphi\Big(B_\Pi,\frac{2n+1}{4},\frac{n+2}{2}\Big)\\
&=\frac{1}{64}n^5 - \frac{23}{256}n^4 - \frac{17}{32}n^3 + \frac{271}{512}n^2 - \frac{1405}{1024}n - \frac{24335}{4096}\\
&>0, 
\end{aligned}
$$
where the last inequality follows from $n\geq 11$. Thus we have $\lambda_1(B_\Pi)<\frac{2n+1}{4}$ or $\lambda_2(B_\Pi)>\frac{2n+1}{4}$. We shall prove that the later case cannot occur.  Let $D=\mathrm{diag}(1,1,d_u-3,1,1,n-1-d_u)$. Then
$$
\begin{aligned}
\tilde{B}_\Pi&=D^{\frac{1}{2}}B_\Pi D^{-\frac{1}{2}}\\
&=
\left[\begin{smallmatrix}
0 &1 &      \sqrt{d_u-3} &    1 &    1 &  0\\
1 &0 &      \sqrt{d_u-3} &    1 &    0 &  0\\
 \sqrt{d_u-3}& \sqrt{d_u-3} &      0 &    0 &    0 &\sqrt{(d_u-3)(n-1-d_u)}\\
1 &1 &     0 &    0 &    1 &  \sqrt{n-1-d_u}\\
1 &0 &       0 &    1 &    0 &  \sqrt{n-1-d_u}\\
 0 &0& \sqrt{(d_u-3)(n-1-d_u)} &\sqrt{n-1-d_u}& \sqrt{n-1-d_u} &  0\\
 \end{smallmatrix}\right]
 \end{aligned}
$$
is symmetric, and has the same eigenvalues as $B_\Pi$. Let $\tilde{B}_\Pi'$ be the matrix obtained by deleting the third row and column from  $\tilde{B}_\Pi$. By the Cauchy interlacing theorem and Lemma \ref{row_sum}, we have
$$
\begin{aligned}
\lambda_2(B_\Pi)=\lambda_2(\tilde{B}_\Pi)&\leq\lambda_1(\tilde{B}_\Pi')\\
&\leq \max\left\{\sqrt{n-1-d_u}+3, 2\sqrt{n-1-d_u}\right\}\\
&\leq \frac{2n+1}{4}
\end{aligned}
$$
because  $d_u\in \{\frac{n+1}{2},\frac{n+2}{2},\frac{n+3}{2},\frac{n+4}{2},\frac{n+5}{2}\}$ by (P1) and $n\geq 11$, as required. Therefore, we conclude that $\rho_A(G)=\lambda_1(B_\Pi)<\frac{2n+1}{4}<\frac{\sqrt{n^2-3}+1}{2}<\rho_A(H_n)$ by Lemma \ref{extremal-graph}, contrary to our assumption. 

If $c=1$, since $G$ is wheel-free, we must have $G=G(d_u-3,1,1,\bar{d}_u-2)-e=G(d_u-3,1,1,n-3-d_u)-e$, where $e$ is an edge between $\{v_1,w_1\}$ and $\{x_1,y_1\}$ in $G(d_u-3,1,1,n-3-d_u)$. By symmetry, we may assume that $e=v_1x_1$ or $w_1y_1$. If $e=v_1x_1$, we see that $A(G)$ has the equitable quotient matrix 
$$
B_\Pi=\begin{bmatrix}
0 &1& d_u-3 &2& 0\\
 1 &0 &d_u-3 &1 &0 \\
  2 &2& 0 &0&  n-3-d_u\\
   2 &1 &0 &1 &n-3-d_u\\
    0 &0 &d_u-3& 2& 0 \\
\end{bmatrix}
\begin{array}{l}
\{u,y_1\}\\
\{u_0,x_1\}\\
\{u_1,\ldots,u_{d_u-3}\}\\
\{v_1,w_1\}\\
\{z_1,\ldots,z_{n-3-d_u}\}
\end{array}
$$
As above, we can deduce that $\lambda_1(B_\Pi)<\frac{2n+1}{4}$ or $\lambda_2(B_\Pi)>\frac{2n+1}{4}$. Again, we claim that the later case cannot occur. In fact, if $n=11$, we can directly verify the result because we have known that $d_u\in \{6,7,8\}$ by (P1). For $n\geq 12$, let 
$$
\begin{aligned}
\tilde{B}_\Pi&=D^{\frac{1}{2}}B_\Pi D^{-\frac{1}{2}}\\
&=
\left[\begin{smallmatrix}
0 &1 &      \sqrt{2(d_u-3)} &    2 &  0\\
1 &0&      \sqrt{2(d_u-3)} &    1 &  0\\
\sqrt{2(d_u-3)} &\sqrt{2(d_u-3)} &     0 &   0 &  \sqrt{(d_u-3)(n-3-d_u)}\\
2 &1 &     0 &   1 &  \sqrt{2(n-3-d_u)}\\
0 &0 &    \sqrt{(d_u-3)(n-3-d_u)} &   \sqrt{2(n-3-d_u)} &  0\\
 \end{smallmatrix}\right],
 \end{aligned}
$$
where $D=\mathrm{diag}(2,2,d_u-3,2,n-3-d_u)$, and let  $\tilde{B}_\Pi'$ denote the matrix obtained by deleting the third row and column from  $\tilde{B}_\Pi$. Then we have
$$
\begin{aligned}
\lambda_2(B_\Pi)=\lambda_2(\tilde{B}_\Pi)&\leq\lambda_1(\tilde{B}_\Pi')\\
&\leq  \sqrt{2(n-3-d_u)}+4\\
&\leq \frac{2n+1}{4}
\end{aligned}
$$
because $d_u\in \{\frac{n+1}{2},\frac{n+2}{2},\frac{n+3}{2},\frac{n+4}{2},\frac{n+5}{2}\}$ and $n\geq 12$, as required. Thus $\rho_A(G)=\lambda_1(B_\Pi)<\frac{2n+1}{4}<\frac{\sqrt{n^2-3}+1}{2}<\rho_A(H_n)$ by Lemma \ref{extremal-graph},  a contradiction. If $e=w_1y_1$, then $A(G)$ has the equitable quotient matrix 
$$
B_\Pi=\begin{bmatrix}
0 &1& d_u-3 &1& 1&0\\
 1 &0 &d_u-3 &1 &0 &0 \\
  2 &2& 0 &0&  0& n-3-d_u\\
   2 &2 &0 &0& 1 &n-3-d_u\\
   2 &0 &0 &1 &0 &n-3-d_u\\
    0 &0 &d_u-3& 1& 1 &0 \\
\end{bmatrix}
\begin{array}{l}
\{u,x_1\}\\
\{u_0,y_1\}\\
\{u_1,\ldots,u_{d_u-3}\}\\
\{v_1\}\\
\{w_1\}\\
\{z_1,\ldots,z_{n-3-d_u}\}
\end{array}
$$
By using a similar analysis as above, we can deduce that $\lambda_1(B_\Pi)<\frac{2n+1}{4}$, and so $\rho_A(G)=\lambda_1(B_\Pi)<\frac{\sqrt{n^2-3}+1}{2}<\rho_A(H_n)$, contrary to the assumption.

 {\flushleft \bf Subcase 1.3.} $2\leq b \leq \frac{d_u-1}{2}$.

 In this situation, we have $G=G(d_u-2b-1,b,c,d)$ or  $G(d_u-2b-1,b,c,d)-e$, where $e$ is some edge between $N(u)$ and $N_2(u)$ in $G(d_u-2b-1,b,c,d)$. Since $G$ is wheel-free, we must have $c=0$. Notice that both $G=G(d_u-2b-1,b,0,d)$ and  $G(d_u-2b-1,b,0,d)-e$ are wheel-free. Thus we conclude that  $G=G(d_u-2b-1,b,0,d)=G(d_u-2b-1,b,0,\bar{d}_u)=G(d_u-2b-1,b,0,n-1-d_u)$. 
 
If $b\neq \frac{d_u-1}{2}$, then $b\leq \frac{d_u-2}{2}$, and we see that  $A(G)$ has the equitable quotient matrix
$$
B_\Pi=\begin{bmatrix}
0 &1& d_u-2b-1 &b &b& 0\\
 1 &0 &d_u-2b-1 &b &0 &0\\
  1 &1& 0 &0& 0& n-1-d_u\\
   1 &1 &0 &0 &1 &n-1-d_u\\
    1 &0 &0& 1& 0 &n-1-d_u\\
     0 &0 &d_u-2b-1 &b &b &0
\end{bmatrix}
\begin{array}{l}
\{u\}\\
\{u_0\}\\
\{u_1,\ldots,u_{d_u-2b-1}\}\\
\{v_1,\ldots, v_b\}\\
\{w_1,\ldots, w_b\}\\
\{z_1,\ldots,z_{n-1-d_u}\}
\end{array}
$$
By a simple calculation, the  characteristic polynomial of $B_\Pi$ is 
$$
\begin{aligned}
\varphi(B_\Pi,x,d_u,b)&=x^6 + (d_u^2 - (n +2)d_u + n+b-1)x^4 + ((2b-2)d_u+2b+2 - 2bn)x^3 \\
&~~~- ((b+2)d_u^2 - (b^2+(b+2)n +3b+2)d_u + (b+1)(b+2)n +4b-1)x^2\\
&~~~+ ((2-2b)d_u+(2n - 6)b-2)x + d_u^2 - (n + 2b)d_u +(2b+1)(n-1).
\end{aligned}
$$
We shall prove that $\lambda_1(B_\Pi)<\frac{2n+1}{4}$.  It is easy to see that $\varphi\left(B_\Pi,\frac{2n+1}{4},d_u,b\right)$ can be expressed as
$$
\varphi\Big(B_\Pi,\frac{2n+1}{4},d_u,b\Big)=\alpha(d_u,n) \cdot b^2 +\beta(d_u,n) \cdot b +\gamma(d_u,n),
$$
where $\alpha(d_u,n)=-\frac{1}{16}(2n + 1)^2(n-d_u)$, $\beta(d_u,n)$ and $\gamma(d_u,n)$ are the functions of $d_n$ and $n$. Since $\alpha(d_u,n)<0$, we must have 
$$\varphi\Big(B_\Pi,\frac{2n+1}{4},d_u,b\Big)\geq\min\left\{\varphi\Big(B_\Pi,\frac{2n+1}{4},d_u,2\Big),\varphi\Big(B_\Pi,\frac{2n+1}{4},d_u,\frac{d_u-2}{2}\Big)\right\}.$$
Recall that $d_u\in \{\frac{n+1}{2},\frac{n+2}{2},\frac{n+3}{2},\frac{n+4}{2},\frac{n+5}{2}\}$ by (P1). If $d_u=\frac{n+1}{2}$, by a simple computation, we obtain 
$$
\varphi\Big(B_\Pi,\frac{2n+1}{4},d_u,2\Big)=\frac{1}{64}n^5 - \frac{19}{256}n^4 - \frac{7}{8}n^3 + \frac{195}{512}n^2 - \frac{4101}{1024}n -\frac{43067}{4096}>0
$$
and 
$$
\varphi\Big(B_\Pi,\frac{2n+1}{4},d_u,\frac{d_u-2}{2}\Big)=\frac{1}{128}n^5 - \frac{7}{256}n^4 - \frac{223}{512}n^3 - \frac{97}{256}n^2 + \frac{1005}{512}n -\frac{11361}{4096}>0
$$
by considering that $n\geq 11$. For other values of $d_u$,  one can verify that the above two inequalities also hold. Thus  $\varphi(B_\Pi,\frac{2n+1}{4},d_u,b)>0$, which leads to $\lambda_1(B_\Pi)<\frac{2n+1}{4}$ or $\lambda_2(B_\Pi)>\frac{2n+1}{4}$. Again, the later case cannot occur. Let 
$$
\begin{aligned}
\tilde{B}_\Pi&=D^{\frac{1}{2}}B_\Pi D^{-\frac{1}{2}}\\
&=
\left[\begin{smallmatrix}
0 &1 &      \sqrt{d_u-2b-1} &   \sqrt{b} &    \sqrt{b}&  0\\
1 &0 &      \sqrt{d_u-2b-1} &    \sqrt{b} &    0 &  0\\
 \sqrt{d_u-2b-1}& \sqrt{d_u-2b-1} &      0 &    0 &    0 &\sqrt{(d_u-2b-1)(n-1-d_u)}\\
\sqrt{b} &\sqrt{b} &     0 &    0 &    1 & \sqrt{b(n-1-d_u)}\\
\sqrt{b} &0 &       0 &    1 &    0 & \sqrt{b(n-1-d_u)}\\
 0 &0& \sqrt{(d_u-2b-1)(n-1-d_u)} &\sqrt{b(n-1-d_u)}& \sqrt{b(n-1-d_u)} &  0\\
 \end{smallmatrix}\right],
 \end{aligned}
$$
where $D=\mathrm{diag}(1,1,d_u-2b-1,b,b,n-1-d_u)$, and let $\tilde{B}_\Pi'$ be the matrix obtained by deleting the last  row and column from  $\tilde{B}_\Pi$. As above, we have 
$$
\begin{aligned}
\lambda_2(B_\Pi)=\lambda_2(\tilde{B}_\Pi)\leq\lambda_1(\tilde{B}_\Pi')\leq \max\left\{\sqrt{d_u-2b-1}+2\sqrt{b}+1, 2\sqrt{d_u-2b-1}\right\}.
\end{aligned}
$$
Notice that 
$$(\sqrt{d_u-2b-1}+2\sqrt{b})^2=d_u+2b-1+4\sqrt{(d_u-2b-1)b}\leq 3d_u-3.$$ Combining this with $d_u\in \{\frac{n+1}{2},\frac{n+2}{2},\frac{n+3}{2},\frac{n+4}{2},\frac{n+5}{2}\}$ and $n\geq 11$, we can  deduce that $\sqrt{d_u-2b-1}+2\sqrt{b}+1\leq \frac{2n+1}{4}$. Also, one can verify that $2\sqrt{d_u-2b-1}\leq 2\sqrt{d_u-5}\leq \frac{2n+1}{4}$. Thus  $\lambda_2(B_\Pi)\leq \frac{2n+1}{4}$, as required. Therefore, we have $\rho_A(G)=\lambda_1(B_\Pi)<\frac{2n+1}{4}<\frac{\sqrt{n^2-3}+1}{2}<\rho_A(H_n)$, which is impossible.

If $b=\frac{d_u-1}{2}$,  then $G=G(0,\frac{d_u-1}{2},0,n-1-d_u)$, and $A(G)$ has the equitable quotient matrix
$$
B_\Pi=\begin{bmatrix}
0 &1& \frac{d_u-1}{2}&\frac{d_u-1}{2} & 0\\
 1 &0 &\frac{d_u-1}{2} &0 &0\\
  1 &1& 0 & 1& n-1-d_u\\
   1 &0 &1 &0  &n-1-d_u\\
     0 &0 &\frac{d_u-1}{2} &\frac{d_u-1}{2} &0
\end{bmatrix}
\begin{array}{l}
\{u\}\\
\{u_0\}\\
\{v_1,\ldots, v_{\frac{d_u-1}{2}}\}\\
\{w_1,\ldots, w_{\frac{d_u-1}{2}}\}\\
\{z_1,\ldots,z_{n-1-d_u}\}
\end{array}
$$
By using a similar method, we also obtain $\rho_A(G)=\lambda_1(B_\Pi)<\frac{2n+1}{4}<\frac{\sqrt{n^2-3}+1}{2}<\rho_A(H_n)$, contrary to the assumption. 

{\flushleft \bf Case 2.}  $p_u=0$.

\begin{table}[t]
\caption{All possible forms of $G=(\lfloor\frac{d_u}{2}\rfloor K_2\cup (d_u-2\lfloor\frac{d_u}{2}\rfloor)K_1)\nabla (n-d_u)K_1$.}
\centering
\setlength{\tabcolsep}{13mm}{
\begin{tabular}{cccc}
\toprule
$n~\mathrm{mod}~4$ & $d_u$ & $G$\\
\midrule
$1$ & $\frac{n+1}{2}$ & $(\frac{n-1}{4}K_2\cup K_1)\nabla \frac{n-1}{2}K_1$\\
$3$ & $\frac{n+1}{2}$ & $\frac{n+1}{4}K_2\nabla \frac{n-1}{2}K_1=H_n$\\
$0$ & $\frac{n}{2}$ & $\frac{n}{4}K_2\nabla \frac{n}{2}K_1=H_n$\\
$0$ & $\frac{n+2}{2}$  & $(\frac{n}{4}K_2\cup K_1)\nabla \frac{n-2}{2}K_1$\\
$2$ & $\frac{n}{2}$  & ($\frac{n-2}{4}K_2\cup K_1)\nabla \frac{n}{2}K_1=H_n$\\
$2$ & $\frac{n+2}{2}$ & $\frac{n+2}{4}K_2\nabla \frac{n-2}{2}K_1$\\
\bottomrule
\end{tabular}}
\label{tab-2}
\end{table}

In this case,  since $G[N(u)]$ is $P_3$-free, we have $G[N(u)]=aK_2\cup bK_1$ with $a,b\geq 0$ and $2a+b=d_u$. First we shall prove that $a>0$. In fact, if $a=0$, i.e., $G[N(u)]$ is an empty graph, then 
$$
 R_u=d_u+2e(G[N(u)])+e(N(u),N_2(u))\leq d_u+d_u\bar{d}_u=-d_u^2+nd_u.
$$
By Claim \ref{claim-1}, we deduce that
$$
d_u^2-nd_u+\frac{(n+1)^2-1}{4}\leq 0,
$$
which is impossible because $\Delta=n^2-((n+1)^2-1)=-2n<0$. Next we claim that  $G[N_2(u)]$ is also $P_3$-free. If not, since $G[N(u)]$ contains at least two vertices due to $a>0$, we have $e(N(u),N_2(u))\leq d_u\bar{d}_u-1$ by Fact \ref{fact-2}, and so 
$$
\begin{aligned}
R_u&=d_u+2(d_u-\omega_u)+e(N(u),N_2(u))\\
&\leq d_u+2(d_u-\frac{d_u}{2})+d_u\bar{d}_u-1\\
&=-d_u^2+(n+1)d_u-1.
\end{aligned}
$$
Combining this with Claim \ref{claim-1} yields that 
$$
d_u^2-(n+1)d_u+1+\frac{(n+1)^2-1}{4}\leq 0,
$$
which is impossible because $\Delta=(n+1)^2-4(1+\frac{(n+1)^2-1}{4})=-3<0$.  Hence, we can suppose  $G[N_2(u)]=cK_2\cup dK_1$, where $c,d\geq 0$ and $2c+d=\bar{d}_u=n-1-d_u$. Then, again by Fact \ref{fact-2}, we have  $e(N(u),N_2(u))\leq d_u\bar{d}_u-ac$, and so
$$
R_u\leq d_u+2\Big(d_u-\frac{d_u}{2}\Big)+d_u\bar{d}_u-ac=-d_u^2+(n+1)d_u-ac.
$$
Combining this with  Claim \ref{claim-1}, we obtain
\begin{equation}\label{eq-4}
d_u^2-(n+1)d_u+ac+\frac{(n+1)^2-1}{4}\leq 0, 
\end{equation}
which implies  that  $c=0$ because $a>0$ and  $\Delta=1-4ac\geq 0$. Putting $c=0$ in (\ref{eq-4}), we obtain $d_u=\frac{n}{2}$, $\frac{n+1}{2}$ or $\frac{n+2}{2}$. Furthermore, according to the above discussions, we must have $e(N(u),e(N_2(u)))=d_u\bar{d}_u$. Concluding these results, we obtain $G=(aK_2\cup bK_1)\nabla (\bar{d_u}+1)K_1=(aK_2\cup bK_1)\nabla (n-d_u)K_1$, where $a>0$ and $2a+b=d_u\in \{\frac{n}{2}, \frac{n+1}{2},\frac{n+2}{2}\}$. Notice that $(aK_2\cup bK_1)\nabla (n-d_u)K_1$ is always wheel-free. By considering the maximality of $\rho_A(G)$, we conclude that $G=(\lfloor\frac{d_u}{2}\rfloor K_2\cup (d_u-2\lfloor\frac{d_u}{2}\rfloor)K_1)\nabla (n-d_u)K_1$ with $d_u\in \{\frac{n}{2}, \frac{n+1}{2},\frac{n+2}{2}\}$. In Table \ref{tab-2}, we list all possible forms of $G$. If $n\equiv 1~\mathrm{mod}~4$, then $G=(\frac{n-1}{4}K_2\cup K_1)\nabla \frac{n-1}{2}K_1$. According to  the analysis of Subcase 1.1, we have $\rho_A(G)=\rho_A((\frac{n-1}{4}K_2\cup K_1)\nabla \frac{n-1}{2}K_1)<\rho_A(H_n)$, contrary to our assumption. If  $n\equiv 3~\mathrm{mod}~4$,  then $G=H_n$, as required. For $n\equiv 0~\mathrm{mod}~4$ and $n\equiv 2~\mathrm{mod}~4$, as in Subcase 1.1, we can verify that the graph $H_n$ always has the maximum spectral radius.

We complete the proof.
\end{proof}

Now we  give the proof of Theorem \ref{thm-2}.

\renewcommand\proofname{\bf Proof of Theorem \ref{thm-2}}
\begin{proof}
Assume that $G$ is a graph attaining the maximum signless Laplacian spectral radius among all wheel-free graphs of order $n$. As in the proof of Theorem \ref{thm-1}, we claim that $G$ is connected.  Since  $K_2 \nabla (n-2)K_1$ is wheel-free, we have 
$$\rho_Q(G)\geq \rho_Q(K_2 \nabla (n-2)K_1)=\frac{n+2+\sqrt{(n+2)^2-16}}{2}$$
by Lemma \ref{Q-extremal-graph}, which gives that 
\begin{equation}\label{eq-5}
\rho_Q^2(G)-(n+2)\rho_Q(G)+4\geq 0.
\end{equation}
Let $Q^*(G)=Q^2(G)-(n+2)Q(G)+4I_n$. Clearly, $\rho_Q^2(G)-(n+2)\rho_Q(G)+4$ is an eigenvalue of $Q^*(G)$ with an eigenvector all of whose entries are nonnegative. Let $R_v^*$ be the row sum of $Q^*(G)$ corresponding to $v\in V(G)$, and let $
R_u^*=\max_{v\in V(G)}R_v^*$.
By Lemma \ref{row_sum} and (\ref{eq-5}), we have 
\begin{equation}\label{eq-6}
R_u^*\geq \rho_Q^2(G)-(n+2)\rho_Q(G)+4\geq 0.
\end{equation}
On the other hand, we see that
\begin{equation}\label{eq-7}
\begin{aligned}
R_u^*&=2d_u^2+2[d_u+2e(G[N(u)])+e(N(u),N_2(u))]-2(n+2)d_u+4\\
&\leq 2d_u^2+2[d_u+2(d_u-\omega_u)+d_u\bar{d_u}]-2(n+2)d_u+4\\
&\leq 2d_u^2+2[d_u+2(d_u-1)+d_u(n-1-d_u)]-2(n+2)d_u+4\\
&=0,
\end{aligned}
\end{equation}
where $\bar{d}_u=|N_2(u)|$ and $\omega_u$ are defined as in the proof of Theorem \ref{thm-1}.
Combining (\ref{eq-6}) and (\ref{eq-7}), we obtain $\omega_u=1$ (i.e., $G[N(u)]$ is a tree), $\bar{d}_u=n-1-d_u$ (i.e., $V(G)=\{u\}\cup N(u)\cup N_2(u)$) and $e(N(u),N_2(u))=d_u\bar{d}_u$ (that is, the edges between $N(u)$ and $N_2(u)$ form a complete bipartite graph). We consider the following  two cases.

{\flushleft \bf Case 1.} $N_2(u)=\emptyset$.

In this situation, from the above arguments we obtain $G=K_1\nabla G[N(u)]$, where $G[N(u)]$ is a tree of order $n-1$. If $G(N[u])=K_{1,n-2}$, then  $G=K_2\nabla (n-2)K_1$, as required. Now suppose $G(N[u])\neq K_{1,n-2}$.  Let $\mathbf{x}$ be the unique unit positive  eigenvector (or Perron vector) of $Q(G)$ corresponding to $\rho_Q(G)$, and let $\mathbf{x}_{v_0}=\max_{v\in N(u)} \mathbf{x}_{v}$. Let $\ell$ denote the diameter of $G[N(u)]$. Notice that $\ell\geq 3$ because $G(N[u])$ is a tree but $G(N[u])\neq K_{1,n-2}$. For $1\leq k\leq \ell$,  let $N_k^*(v_0)$ denote the set of vertices at distance $k$ from $v_0$ in $G[N(u)]$. Since $G[N(u)]$ is a tree, we observe that $N_k^*(v_0)$ is an independent set of $G[N(u)]$, and each vertex of $N_k^*(v_0)$ has exactly one neighbor in $N_{k-1}^*(v_0)$.  Let $G'$ be the graph obtained from $G$ by deleting those edges not incident with $v_0$ and connecting $v_0$ with all the resulting isolated vertices in $G[N(u)]$. It is clear that  $G'=K_2\nabla (n-2)K_1$. Then we have
$$
\begin{aligned}
\rho_Q(G)&=\mathbf{x}^TQ(G)\mathbf{x}\\
&= \sum_{vw\in E(G)}(\mathbf{x}_v+\mathbf{x}_w)^2\\
&= \sum_{v\in N(u)}(\mathbf{x}_u+\mathbf{x}_v)^2+\sum_{vw\in E(G[N(u)])}(\mathbf{x}_v+\mathbf{x}_w)^2\\
&=\sum_{v\in N(u)}(\mathbf{x}_u+\mathbf{x}_v)^2+\sum_{w\in N(v_0)}(\mathbf{x}_{v_0}+\mathbf{x}_w)^2+\sum_{i=1}^{\ell-1}\sum_{v_iv_{i+1}\in E(N_i^*(v_0),N_{i+1}^*(v_0))}(\mathbf{x}_{v_i}+\mathbf{x}_{v_{i+1}})^2\\
&\leq \sum_{v\in N(u)}(\mathbf{x}_u+\mathbf{x}_v)^2+\sum_{w\in N(v_0)}(\mathbf{x}_{v_0}+\mathbf{x}_w)^2+\sum_{i=1}^{\ell-1}\sum_{v_iv_{i+1}\in E(N_i^*(v_0),N_{i+1}^*(v_0))}(\mathbf{x}_{v_0}+\mathbf{x}_{v_{i+1}})^2\\
&=\sum_{v\in N(u)}(\mathbf{x}_u+\mathbf{x}_v)^2+\sum_{w\in N(v_0)}(\mathbf{x}_{v_0}+\mathbf{x}_w)^2+\sum_{i=1}^{\ell-1}\sum_{v_{i+1}\in N_{i+1}^*(v_0)}(\mathbf{x}_{v_0}+\mathbf{x}_{v_{i+1}})^2\\
&= \sum_{v\in N(u)}(\mathbf{x}_u+\mathbf{x}_v)^2+\sum_{vw\in E(G'[N(u)])}(\mathbf{x}_v+\mathbf{x}_w)^2\\
&= \sum_{vw\in E(G')}(\mathbf{x}_v+\mathbf{x}_w)^2\\
&=\mathbf{x}^TQ(G')\mathbf{x}\\
&\leq \rho_Q(G').
\end{aligned}
$$
We claim that $\rho_Q(G)<\rho_Q(G')=\rho_Q(K_2\nabla (n-2)K_1)$. In fact, if $\rho_Q(G)=\rho_Q(G')$, then from the above inequality and the Perron-Frobenius theorem we see that $\mathbf{x}$ is an eigenvector of $Q(G')$ corresponding to $\rho_Q(G')$ and  $\mathbf{x}_v=\mathbf{x}_{v_0}$ for any non-pendant vertex $v$ of $G[N(u)]$. Let $v$ be a non-pendant neighbor of $v_0$ in $G[N(v)]$ (such a vertex exists because $\ell\geq 3$). Notice that $v$ is of degree $1$ in $G'[N(u)]$. By considering the eigenvalue-eigenvector equation of  $\rho_Q(G')$ and $\mathbf{x}$ at $v_0$ and $v$, we have 
$$
\left\{
\begin{aligned}
\rho_Q(G') \mathbf{x}_{v_0}&=(n-1) \mathbf{x}_{v_0}+\sum_{w\in N(u)\setminus\{v_0\}}\mathbf{x}_{w}+x_u,\\
\rho_Q(G') \mathbf{x}_{v}&=2 \mathbf{x}_{v}+\mathbf{x}_{v_0}+x_u,
\end{aligned}
\right.
$$
which is impossible because $n\geq 4$, $v\in N(u)\setminus\{v_0\}$ and $\mathbf{x}_v=\mathbf{x}_{v_0}$.
Therefore, in this situation, we conclude that  $G=K_2\nabla (n-2)K_1$. 

{\flushleft \bf Case 2.} $N_2(u)\neq \emptyset$.

In this situation, we see that $G[N(u)]$ is $P_3$-free because $G$ is wheel-free. Thus $G[N(u)]=K_1$ or $K_2$ by the above arguments. If $G[N(u)]=K_1$, then $d_u=1$. Let $v_0$ be the unique neighbor of $u$. Then  $d_{v_0}=n-1$. Since $G[N(v_0)]$ must be a forest (with at least two components), according to the proof of Case 1,  we conclude that $\rho_Q(G)<\rho_Q(K_2\nabla (n-2)K_1)$. If $G[N(u)]=K_2$, then $N_2(u)$ must be an independent set by Fact \ref{fact-2}, and so we obtain $G=K_2\nabla (n-2)K_1$.

We complete the proof.
\end{proof}

\section*{Acknowledgements}
The authors are indebted to S.M. Cioab\u{a}, Zhiwen Wang and Zhenzhen Lou for many helpful suggestions. This research was supported by the National Natural Science Foundation of China (Nos. 11901540, 11671344 and 11771141).


\begin{thebibliography}{99}
\setlength{\itemsep}{0pt}

\bibitem{BG}  L. Babai, B. Guiduli, Spectral extrema for graphs: the Zarankiewicz problem, Electron. J. Combin. 16  (2009) \#R123.

\bibitem{BH} A.E. Brouwer, W.H. Haemers, Spectra of Graphs, Springer, Berlin, 2011.

\bibitem{BS} R.A. Brualdi, E.S. Solheid, On the spectral radius of complementary acyclic matrices of zeros and ones, SIAM J. Algebraic Discr. Methods 7 (1986) 265--272.

\bibitem{CLZ}  M.Z. Chen, A.M. Liu, X.D. Zhang, Spectral extremal results with forbidding linear forests, Graphs Combin. 35 (2019) 335--351.

\bibitem{CFTZ}  S.M. Cioab\u{a}, L. Feng, M. Tait, X.D. Zhang, The spectral radius of graphs with no intersecting triangles, 2019, \href{https://arxiv.org/abs/1911.13082}{arXiv: 1911.13082}.


\bibitem{CF} D. Conlon, J. Fox, Graph removal lemmas, Surveys in combinatorics 2013, 1--49, London Math. Soc. Lecture Note Ser., 409, Cambridge Univ. Press, Cambridge, 2013.

\bibitem{DZ} T. Dzido, A note on Tur\'{a}n numbers for even wheels, Graphs Combin. 29 (2013) 1305--1309.

\bibitem{DJ} T. Dzido, A. Jastrz\c{e}bski, Tur\'{a}n numbers for odd wheels, Discrete Math. 341 (4) (2018) 1150--1154.

\bibitem{EZ} M.N. Ellingham, X. Zha, The spectral radius of graphs on surfaces, J. Combin. Thoery Ser. B 78 (2000) 45--56.

\bibitem{GH}  J. Gao, X. Hou, The spectral radius of graphs without long cycles, Linear Algebra Appl. 566
(2019) 17--33.

\bibitem{GR} C. Godsil, G. Royle, Algebraic Graph Theory, Graduate Texts in Mathematics, 207, Springer-Verlag, New York, 2001.


\bibitem{LLT} M. Lu, H. Liu, F. Tian, A new upper bound for the spectral radius of graphs with girth at least $5$, Linear Algebra Appl. 414 (2006) 512--516.

\bibitem{NI1}  V. Nikiforov, Bounds on graph eigenvalues II, Linear Algebra Appl. 427 (2007) 183--189.
 
 \bibitem{NI2} V. Nikiforov, A spectral condition for odd cycles in graphs, Linear Algebra Appl. 428 (2008)
1492–1498.
 
\bibitem{NI3}  V. Nikiforov, The spectral radius of graphs without paths and cycles of specified length, Linear Algebra Appl. 432 (2010) 2243--2256.

\bibitem{NI4}  V. Nikiforov, A contribution to the Zarankiewicz problem, Linear Algebra Appl. 432 (2010)  1405--1411.

\bibitem{NI5}  V. Nikiforov, Some new results in extremal graph theory, Surveys in Combinatorics 2011, 141--181, London Math. Soc. Lecture Note Ser., 392, Cambridge Univ. Press, Cambridge, 2011.

\bibitem{SS}  L. Shi, Z. Song, Upper bounds on the spectral radius of book-free and/or $K_{2,r+1}$-free graphs, Linear Algebra Appl. 420 (2007) 526--529.

\bibitem{SI} A. Sidoreno, What we know and what we do not know about Tur\'{a}n numbers, Graphs Combin. 11 (2) (1995) 179--199.

\bibitem{SIM} M. Simonovits, A method for solving extremal problems in graph theory, stability problems, Theory of Graphs (Proc. Colloq., Tihany, 1966) (1968) 279--319.

\bibitem{ST} W.A. Stein et al., SageMath, The Sage Mathematics Software System (Version
9.1), 2020. \url{http://www.sagemath.org}.

\bibitem{STE} D. Stevanovi\'{c}, Spectral Radius of Graphs, Academic Press, London, 2015. 

\bibitem{WI} H. Wilf, Spectral bounds for the clique and independence numbers of graphs, J. Combin. Theory Ser. B 40 (1986) 113--117.

\bibitem{YU} L.T. Yuan, Extremal graphs for wheels, 2020, \href{https://arxiv.org/abs/2001.02628}{arXiv: 2001.02628}.

\bibitem{YWZ}  W. Yuan, B. Wang, M. Zhai, On the spectral radii of graphs without given cycles, Electron. J. Linear Algebra 23 (2012) 599--606.


\bibitem{ZW} M. Zhai, B. Wang, Proof of a conjecture on the spectral radius of $C_4$-free graphs, Linear Algebra Appl. 437 (2012) 1641--1647.

\bibitem{ZWF} M. Zhai, B. Wang, L. Fang, The spectral Tur\'{a}n problem about graphs with no $6$-cycle, Linear Algebra Appl. 590 (2020), 22--31.

\bibitem{ZL} M. Zhai, H. Lin, Spectral extrema of graphs: Forbidden hexagon, Discrete Math. 343 (10) (2020) 112028.


\end{thebibliography}
\end{document}